\documentclass{aims} % Use the aims.cls file to compile your paper
\usepackage{amsmath}
\usepackage{paralist}
\usepackage[misc]{ifsym}
\usepackage{epsfig} %For pictures: screened artwork should be set up with an 85 or 100 line screen
\usepackage{epstopdf} %This is to transfer .eps figure to .pdf figure; please compile your article using PDFLaTex or PDFTeXify.
\usepackage[colorlinks=true]{hyperref}
\hypersetup{urlcolor=blue, citecolor=red}
\allowdisplaybreaks

\def\currentvolume{X}
 \def\currentissue{X}
  \def\currentyear{200X}
   \def\currentmonth{XX}
    \def\ppages{X-XX}
     \def\DOI{10.3934/xx.xxxxxxx}

\newtheorem{theorem}{Theorem}[section]

\newtheorem{lemma}[theorem]{Lemma}
\newtheorem{proposition}[theorem]{Proposition}

\theoremstyle{definition}
\newtheorem{definition}[theorem]{Definition}
\newtheorem{remark}[theorem]{Remark}

\newcommand{\R}{\mathbb{R}}

\title[Running head]
{Infinitely many solutions for a class of fractional Schr\"{o}dinger
 equations coupled with neutral
 scalar field} % Only the first word and proper nouns should be capitalized.

\author[Liejun Shen, Marco Squassina and Xiaoyu Zeng]{}

\subjclass{Primary: 35J60,~35Q55; Secondary: 53C35.}
\keywords{Fractional Schr\"{o}dinger equation, Strongly indefinite, Bessel operator,
Fountain theorem, Infinitely many soultions, Variational methods.}

\thanks{L. Shen is partially supported by NSFC (Grant No. 12201565).
M.\  Squassina  is  member  of  Gruppo  Nazionale  per
	l'Analisi  Matematica,  la Probabilita  e  le  loro  Applicazioni  (GNAMPA)  of  the  Istituto  Nazionale  di  Alta  Matematica  (INdAM).
X. Zeng is supported by CSC and NSFC (Grant Nos.
 12322106, 12171379, 12271417).
}

\thanks{$^*$Corresponding author: Marco Squassina}

\begin{document}
\maketitle

% Enter the first author's name and email address; email addresses are required for each author.
% Use footnote notations to indicate address and affiliations with commas between numbers if more than one address applies; see below for examples.
\centerline{\scshape
Liejun Shen$^{{\href{mailto:ljshen@zjnu.edu.cn}{\textrm{\Letter}}}1}$,
Marco Squassina$^{{\href{mailto:marco.squassina@unicatt.it}{\textrm{\Letter}}}*2}$ and
Xiaoyu Zeng$^{{\href{mailto:xyzeng@whut.edu.cn}{\textrm{\Letter}}}3}$}

\medskip

{\footnotesize
% Enter the full affiliation and country name:
% Do not insert commas or periods at the end of lines.
 \centerline{$^1$Department of Mathematics, Zhejiang
Normal University, Jinhua, Zhejiang, 321004, People's Republic of China}
} % Do not forget to end {\footnotesize with the sign }

\medskip

{\footnotesize
 % Enter the full affiliation and country name:
 \centerline{$^2$Dipartimento di Matematica e Fisica
	Universit\`a Cattolica del Sacro Cuore,
	Via della Garzetta 48, 25133, Brescia, Italy}

\medskip

{\footnotesize
 % Enter the full affiliation and country name:
 \centerline{$^3$Centre for Mathematical Sciences,
Wuhan University of Technology, Wuhan, Hubei, 430070, People's Republic of China}
}

}

\bigskip

% The name of the handling editor will be entered by AIMS production staff.
% "Communicated by Handling Editor" is not needed for special issue.
 \centerline{(Communicated by Handling Editor)}

%%%%%%%%%%%%%%%%%%%%%%%%%%%%%%%%%%%%%%%%%%%%%%%%%%%%%%%
%             5. ABSTRACT
%%%%%%%%%%%%%%%%%%%%%%%%%%%%%%%%%%%%%%%%%%%%%%%%%%%%%%%

\begin{abstract}
We study the fractional Schr\"{o}dinger
 equations coupled with a neutral
 scalar field
 $$
 \left\{%
\begin{array}{ll}
   (-\Delta)^s u+V(x)u=K(x)\phi u +g(x)|u|^{q-2}u, & x\in \mathbb{R}^3, \\
    (I-\Delta)^t \phi=K(x)u^2, &  x\in \mathbb{R}^3,\\
\end{array}%
\right.
 $$
where $(-\Delta)^s$ and $(I-\Delta)^t$ denote the fractional Laplacian and Bessel operators with
$\frac{3}{4} <s<1$ and $0<t<1$, respectively. Under some suitable assumptions for the external potentials
$V$, $K$ and $g$, given $q\in(1,2)\cup(2,2_s^*)$ with
$2_s^*:= \frac{6}{3-2s}$,
 with the help of an improved Fountain theorem dealing with a class of strongly
indefinite variational problems approached by
Gu-Zhou [Adv. Nonlinear Stud., {\bf17} (2017), 727--738],
 we show that the system admits infinitely many nontrivial solutions.
\end{abstract}

%%%%%%%%%%%%%%%%%%%%%%%%%%%%%%%%%%%%%%%%%%%%%%%%%%%%%%
%                   6. BODY
%%%%%%%%%%%%%%%%%%%%%%%%%%%%%%%%%%%%%%%%%%%%%%%%%%%%%%

% Only the first word and proper nouns of section titles should be capitalized.
% The title of section 1:
\section{Introduction and main results}

\subsection{General overview}
In the present paper, we are concerned with the following fractional elliptic system
\begin{equation}\label{mainequation}
  \left\{%
\begin{array}{ll}
   (-\Delta)^s u+V(x)u=K(x)\phi u +g(x)|u|^{q-2}u, & x\in\R^3, \\
    (I-\Delta)^t \phi=K(x)u^2, &  x\in\R^3,\\
\end{array}%
\right.
\end{equation}
where $(-\Delta)^s$ and $(I-\Delta)^t$ denote the classic fractional Laplacian and Bessel operators with
$\frac{3}{4} <s<1$ and $0<t<1$, respectively.

In light of its relevance in physics, the following
nonlinear fractional Schr\"{o}dinger equation
\begin{equation}\label{fSe}
i \frac{\partial\Psi}{\partial t}= (-\Delta)^{s}\Psi+(V(x)+E)\Psi-f(x,\Psi)\quad\text{for all }x\in\mathbb{R}^{N},
\end{equation}
 where $N>2s$ with $s\in(0,1)$, $E\in\mathbb{R}$, $V$ and $f$ are continuous functions,
 has been received more and more attentions in recent years by a great many mathematicians.
Generally, when they are searching for a particular type of the solutions of
 Eq. \eqref{fSe}, the so called \emph{standing wave solution}, which carries a form of the type
 $$
 \Psi(z,t)=\exp(-iEt)u(z),
 $$
reveals that $u$ acts as a solution of the fractional elliptic equation
\begin{equation}\label{Se}
 \left.\left\{\begin{matrix} (-\Delta)^su+V(x)u=f(x,u)&\text{in}~\mathbb{R}^N,\\
 u\in H^s(\mathbb{R}^N),~u>0,&\text{on}~\mathbb{R}^N.\end{matrix}\right.\right.
\end{equation}
 In the local case, that is, the general semilinear elliptic equations
 \eqref{Se} with $s = 1$, have been extensively considered,
for example, we shall refer the reader to \cite{Rabinowitz,SW2,AS} and their references therein.

In the nonlocal case, namely when $s\in (0, 1)$, the corresponding results for Eq. \eqref{Se}
do never
seem to be as fruitful as the local ones. This occurs
maybe because the techniques and arguments developed for local case cannot be adapted immediately,
c.f. \cite{SSecchi}. In order to introduce some results clearly for Eq. \eqref{fSe},
we recall that, for any $s\in(0, 1)$, the fractional Sobolev space $H^s(\R^N)$ is defined by
$$
H^s(\mathbb{R}^N)=\left\{u\in L^2(\mathbb{R}^N):\int_{\mathbb{R}^N}\int_{\mathbb{R}^N}\frac{(u(x)-u(y))^2}{|x-y|^{N+2s}}dxdy<\infty\right\},
$$
equipped with the norm
$$
\|u\|_{H^s(\mathbb{R}^N)}=\left(\|u\|_{L^2(\mathbb{R}^N)}^2+
\int_{\mathbb{R}^N}\int_{\mathbb{R}^N}\frac{(u(x)-u(y))^2}{|x-y|^{N+2s}}dxdy\right)^{1/2}.
$$
The fractional Laplacian, $(-\Delta)^su$ , of a smooth function $u:\mathbb{R}^N\to\mathbb{R}$, with sufficient decay, is defined by
$$
\mathcal{F}((-\Delta)^{s}u)(\xi)=|\xi|^{2s}\mathcal{F}(u)(\xi),\quad\xi\in\mathbb{R}^{N},
$$
where $\mathcal{F}$ denotes the Fourier transformation which is
$$
{\mathcal F}(\phi)(\xi)=\frac{1}{(2\pi)^{\frac{N}{2}}}\int_{\mathbb{R}^{N}}e^{-i\xi\cdot x}\phi(x)dx:=\widehat{\phi}(\xi),
$$
for functions $\phi$ belonging to Schwartz class.
In reality, according to \cite[Lemma 3.2]{DiNezzaPalatucciValdinoci}
the fractional Laplacian operator can be equivalently represented  as
$$
(-\Delta)^{s}u(x)=-\frac{1}{2}C(N,s)\int_{\mathbb{R}^{N}}\frac{u(x+y)+u(x-y)-2u(x)}{|y|^{N+2s}}dy,\quad\forall x\in\mathbb{R}^{N},
$$
 where
$$
C(N,s)=\left(\int_{\mathbb{R}^{N}}\frac{(1-\cos\xi_{1})}{|\xi|^{N+2s}}d\xi\right)^{-1},\quad
\xi=(\xi_{1},\xi_{2},\ldots,\xi_{N})\in\mathbb{R}^{N}.
$$
Also, due to \cite[Propostion 3.4, Propostion 3.6]{DiNezzaPalatucciValdinoci}, it holds that
$$\|(-\Delta)^{\frac{s}{2}}u\|_{L^2(\mathbb{R}^N)}^2=\int_{\mathbb{R}^N}|\xi|^{2s}|\widehat{u}|^2d\xi=\frac{1}{2}C(N,s)
\int_{\mathbb{R}^N}\int_{\mathbb{R}^{N}}\frac{(u(x)-u(y))^2}{|x-y|^{N+2s}}dxdy,$$
for all $u\in H^s(\mathbb{R}^N)$. Moreover, we usually identify these two quantities by omitting the
normalization constant $\frac{1}{2}C(N,s)$ just for simplicity.
The homogeneous fractiotal Sobolev space $D^{s,2}(\mathbb{R}^N)$ is defined by
$$
D^{s,2}(\mathbb{R}^{N})=\left\{u\in L^{2_{s}^{*}}(\mathbb{R}^{N}):|(-\Delta)^{\frac{s}{2}}u|\in L^{2}(\mathbb{R}^{N})\right\}
$$
which is the completion of $C_0^\infty(\mathbb{R}^N)$ under the norm
$$
\|u\|_{D^{s,2}(\mathbb{R}^{N})}=\left(\int_{\mathbb{R}^{N}}|(-\Delta)^{\frac s2}u|^{2}dx\right)^{\frac12}.
$$
For $N>2s$, from \cite[Theorem 6.5]{DiNezzaPalatucciValdinoci},
 we further know that, for any $p\in[2,2_s^*]$, there exists a constant
$C_p>0$ such that
$$
\|u\|_{L^p(\mathbb{R}^N)}\leq C_p\|u\|_{H^s(\mathbb{R}^N)},~\text{for all}~u\in H^s(\mathbb{R}^N).
$$
Besides, the imbedding $H^s(\mathbb{R}^N) \hookrightarrow L^p_{\text{loc}}(\mathbb{R}^N)$
is compact for all $1\leq p<2_s^*$.

Actually, problem \eqref{Se} was initially proposed by the author in \cite{Laskin1,Laskin2} as a result
of expanding the Feynman path integral, from the Brownian-like to the L\'{e}vy-like
quantum mechanical paths. According to the celebrated paper \cite{Caffarelli},  Eq. \eqref{fSe} and its variants
has been widely contemplated by many authors, specially on the existence of ground state solutions,
positive solutions, sign-changing solutions and multiplicity of standing wave solutions.

For $t>0$, the so-called \emph{Bessel function space} in $\R^3$ is defined by
\[
L^{t,2}(\R^3):=\left\{f\in L^2(\R^3): f=G_t\ast h~\text{for some}~h\in L^2(\R^3)
\right\},
 \]
 where the Bessel convolution kernel is
\begin{equation}\label{Yukawa}
  G_{t}(x):=  \frac{1}{(4\pi)^{\frac{t}{2}} \Gamma(\frac{t}{2})}
\int_{0}^{\infty}\exp\bigg(
-\frac{\pi}{\delta}|x|^2
\bigg)\exp\bigg(-\frac{\delta}{4\pi}\bigg)\delta^{\frac{t-5}{2}}d\delta.
\end{equation}
The operator $(I-\Delta)^{-t}u=G_{2t}\ast u$ is generally known as
Bessel operator of order $t$ and it induces Bessel function space
equipped with the norm $\|f\|_{L^{t,2}(\R^3)}=\|h\|_{L^2(\R^3)}$
if $f=G_t\ast h$. Owing to the view point of Fourier transformation,
this same operator can also read
$$
 G_t=\mathcal{F}^{-1}\circ\left(\left(1+|\xi|^2\right)^{-\frac{t}{2}}\circ\mathcal{F}\right),
$$
so that
$$\|f\|_{L^{t,2}(\R^3)}=\left\|(I-\Delta)^{\frac{t}{2}}f\right\|_{L^2(\R^3)}.
$$
We refer the interested reader to \cite{AdamsHedberg,Stein} and the references therein
for more detailed information
concerning the Bessel operator and Bessel function space.

It should be mentioned here that authors in \cite{FallFelli} introduced
the pointwise formula
$$
(I-\Delta)^{t}u(x)=c_{s}\: \int_{\mathbb{R}^{3}}\frac{u(x)-u(y)}{|x-y|^{\frac{3+2t}{2}}}\mathcal{K}_{\frac{3+2t}{2}}(|x-y|)\:dy+u(x)
$$
for $u\in C_c^2(\R^3)$,
where $K_\nu$ is the modified Bessel
function of the second kind with order $\nu$. Nevertheless, a
closed formula for $K_\nu$ remains unknown, see e.g. \cite[Remark 7.3]{FallFelli}.

Very recently, Felmer-Vergara \cite{Felmer2}
investigated the existence of positive solutions for
fractional equations involving a Bessel operator
\begin{equation}\label{Bessel}
 (I-\Delta)^s u+V(x)u=f(x,u),~ x\in\R^N.
\end{equation}
Subsequently, under different technical assumptions on $V$ and $f$ owing to variational methods,
there are some bibliographies in the study of \eqref{Bessel}, see
\cite{Secchi1,Ikoma,Secchi2} and their references therein for example.

In \cite{Shen}, the author contemplated
the multiplicity and concentration of nontrivial solutions
for the following fractional Schr\"{o}dinger-Poisson
 system involving a Bessel operator
\[
\left\{%
\begin{array}{ll}
    (I-\Delta)^s u+\lambda V(x)u+\phi u= f(x,u)+g(x)|u|^{q-2}u, & x\in\R^3, \\
    (-\Delta)^t \phi=u^2, &  x\in\R^3,\\
\end{array}%
\right.
\]
where $\lambda V$ represents as a deepening potential with $\lambda>0$,
$f\in C^0(\R^3\times\R)$ satisfies some suitable conditions,
$g>0$ is a weight function and $1<q<2$.

Motivated by the relevance of problem \eqref{Se} and the mathematical point of view,
we tend to consider a class of Schr\"{o}dinger
equations coupled with a Bessel operator.
More precisely, we shall establish the existence of infinitely many nontrivial solutions for
the system \eqref{mainequation} under some suitable assumptions on $V$, $K$ and $g$.
Up to the best knowledge of us, it seems the first time to dispose of such type of problems.
What's more, we anticipate that the results in this  paper would prompt the further studies on (fractional)
Schr\"{o}dinger-Poisson systems.

\subsection{Assumptions and  main results}
To arrive at the aim mentioned above, we shall suppose that
\begin{itemize}
   \item[$(H_1)$] $V\in C^0(\R^3)\cap L^\infty(\R^3)$ and 0 lies in a spectrum gap of the operator
   $(-\Delta)^s+V$;
   \item[$(H_2)$] $K\in L^{\frac{3}{4s-3}}(\R^3)\cap L^\infty(\R^3)$ and
   $K(x)>0$ a.e. for $x\in\R^3$;
 \end{itemize}

 Now, the main results in this paper can be stated as follows.

 \begin{theorem}\label{maintheorem1}
 Let $\frac{3}{4}<s<1$ and $0<t<1$ and
suppose that $(H_1)-(H_2)$.
If in addition
\begin{itemize}
   \item[$(H_3)$] $g\in L^{q_0}(\R^3)$ with
   $q_0=\frac{6}{6-q(3-2s)}$ and $1<q<\frac{4}{3}$,
\end{itemize}
then system \eqref{mainequation} possesses infinitely many solutions
 $\{u_n,\phi_{u_n}\}\subset H^s(\R^3)\times L^{t,2}(\R^3)$ satisfying
 $\lim\limits_{n\to\infty}\|u_n\|_{H^s(\R^3)}
 =+\infty$.
\end{theorem}

\begin{theorem}\label{maintheorem2}
 Let $\frac{3}{4}<s<1$ and $0<t<1$ and
suppose that $(H_1)-(H_2)$.
If in addition
\begin{itemize}
  \item[$(H_4)$] $g\in L^{q_0}(\R^3)$ with $g(x)<0$ a.e. $x\in\R^3$,
   $q_0=\frac{6}{6-q(3-2s)}$ and $\frac{4}{3}\leq q<2$;
\end{itemize}
then system \eqref{mainequation} possesses infinitely many solutions
 $\{u_n,\phi_{u_n}\}\subset H^s(\R^3)\times L^{t,2}(\R^3)$ satisfying
 $\lim\limits_{n\to\infty}\|u_n\|_{H^s(\R^3)}
 =+\infty$.
\end{theorem}

\begin{theorem}\label{maintheorem3}
 Let $\frac{3}{4}<s<1$ and $0<t<1$ and
suppose that $(H_1)-(H_2)$.
If in addition
\begin{itemize}
 \item[$(H_5)$] $g\in  L^{q_0}(\R^3)$ with $g(x)>0$ a.e. $x\in\R^3$,
   $q_0=\frac{6}{6-q(3-2s)}$ and
   $2<q<2_s^*$.
\end{itemize}
then system \eqref{mainequation} possesses infinitely many solutions
 $\{u_n,\phi_{u_n}\}\subset H^s(\R^3)\times L^{t,2}(\R^3)$ satisfying
 $\lim\limits_{n\to\infty}\|u_n\|_{H^s(\R^3)}
 =+\infty$.
\end{theorem}

\begin{remark}
Although
the result in Theorem \ref{maintheorem1}
is similar to \cite[Theorem 1.3]{Gu}, as far as we are concerned, one can never prove
Theorem \ref{maintheorem1} by simply repeating the arguments exploited in
the cited paper caused by the appearance of Bessel operator,
see Lemma \ref{Psi} below for example. On the other hand, with the help of some new analytic skills,
the results are even new for the counterparts in \cite{Gu}.
\end{remark}

\begin{remark}
It is
worthy pointing out that we do not conclude whether the results in Theorems \ref{maintheorem2} and \ref{maintheorem3} remain valid when
 $g$ is sign-changing in $(H_4)$ and $(H_5)$, respectively. Moreover, inspired by \cite{SW2,SW3},
 it is interesting to
consider that problem \eqref{mainequation} has a ground state solution.
We postpone these two questions in a further work.
\end{remark}

Again the results in Theorems \ref{maintheorem1}, \ref{maintheorem2} and \ref{maintheorem3}
are new up to now.
To conclude this section, we sketch our proof.
First of all, because the operator $L=(-\Delta )^s+V$
is strongly indefinite, we then follow the idea introduced in \cite{Pankov,SW2,SW3}
to decompose the space $H^s(\R^3)$ suitably. In the meantime, there exist some standard arguments exhibited in Section \ref{Sec2}
that allows us to treat the problem \eqref{mainequation} by variational methods.
Then, we shall depend heavily on a new type of Fountain theorem approached by Gu and Zhou in \cite{Gu}
to derive the existence of infinitely many nontrivial critical points.
Finally, we concentrate on verifying the necessary properties of the corresponding variational functional, see Sections
\ref{Sec2} and \ref{Sec3} in detail. So, we could derive the proofs successfully.
Alternatively, owing to the Bessel operator appearing in \eqref{mainequation},
there are some unpleasant barriers in the last step and we have to take some careful and deep analysis there.

\subsection{Organization of the paper}
The paper is organized as follows. In Section \ref{Sec2}, we provide some preliminary results. Section \ref{Sec3}
is devoted to the proofs of Theorems \ref{maintheorem1}, \ref{maintheorem2} and \ref{maintheorem3}.
\subsection*{Notations}
From now on in the present article, otherwise mentioned particularly, we shall adopt the following notations:
\begin{itemize}
	\item   $C,C_1,C_2,\cdots$ denote any positive constant, whose value is not relevant.
	 	\item      Let $(X,\|\cdot\|_X)$ be a Banach
space with dual space $(X^{-1},\|\cdot\|_{X^{-1}})$, and $\Phi$ be a functional on $X$.
\item Let $\|\cdot\|_{L^p(\R^3)}$ denote the usual $L^p$-norm for any Lebesgue measurable function $u:\R^3\to\R$, where $p\in[1,\infty]$.
	\item Palais-Smale sequence at level $c\in\R$ ($(PS)_c$ sequence in short)
corresponding to a functional $\Phi$ on $X$ means that $\Phi(x_n)\to c$ and $\Phi^{\prime}(x_n)\to 0$ in $X^{-1}$ as $n\to\infty$, where
$\{x_n\}\subset X$.
	\item  If for each $(PS)_c$ sequence $\{x_n\}$ in $X$,
there exists a subsequence $\{x_{n_{k}}\}$ such that $x_{n_{k}}\to x_0$ in $X$ for some $x_0\in X$, then one
 says that the functional $\Phi$ satisfies the so called $(PS)_c$ condition.
\item $o_{n}(1)$ denotes the real sequence with $o_{n}(1)\to 0$
 as $n \to +\infty$.
\item $``\to"$ and $``\rightharpoonup"$ stand for the strong and
 weak convergence in the related function spaces, respectively.
\end{itemize}

\section{Preliminary Results}\label{Sec2}
In this section, we introduce some preliminary results.
For the potential $V\in C^0(\R^3)\cap L^\infty(\R^3)$,
we can define an operator $L:= (-\Delta)^s +V$. Thanks to the
celebrated results in \cite[Theorem 4.26]{Egorov},
one sees that $L$ is self-disjoint with domain $\mathcal{D}(L)=H^s(\R^3)$.
Assume $|L|$ and $|L|^{1/2}$ are the absolute values of $L$ and the square root of $|L|$, respectively.
We denote $\{\mathcal{E}(\lambda):-\infty<\lambda<+\infty\}$
by the spectral family with respective to $L$. Setting
  $U:=\text{id}-\mathcal{E}(0)-\mathcal{E}(0^-)$,
then by virtue of \cite[Theorem IV 3.3]{EE}, $U$ commutes with $L$, $|L|$
and $|L|^{1/2}$. Moreover, $L=U|L|$ is the polar decomposition of $L$.
In view of \cite{Pankov,SW2,SW3}, there holds
\[
X=\mathcal{D}(|L|^{1/2}),~Y =\mathcal{E}(0^-)X~\text{and}~Z =[\text{id}-\mathcal{E}(0)]X.
\]
Via $(V)$, one has $X=Y\bigoplus Z$.
Given $u\in X$, then $u=u^++u^-$ with
\[
u^-=\mathcal{E}(0^-)u:= Pu~\text{and}~u^+ =[\text{id}-\mathcal{E}(0)]u:=Qu.
\]
Furthermore, for all $u\in X\cap \mathcal{D}(L)$,
 one also has that
\begin{equation}\label{settings1}
Lu^-=-|L|u^-~\text{and}~Lu^+=|L|u^+.
\end{equation}
With the above facts in hands, it permits us to introduce an inner product
which could induce the norm on $X$ as follows
\[
(u,v)_X=(|L|^{1/2}u,|L|^{1/2}v)_{L^2(\R^3)}~\text{and}~ \|u\|_X=\big \| |L|^{1/2}u  \big \|_{L^2(\R^3)},
\]
 where $(\cdot,\cdot)_{L^2(\R^3)}$ stands for the usual inner product of $L^2(\R^3)$. Besides, by \eqref{settings1}, we have that
\begin{equation}\label{decomposition}
\int_{\R^3}[|(-\Delta)^{\frac{s}{2}}u|^2+V(x)|u|^2]dx=\|Qu \|^2_X-\|Pu \|^2_X,~\forall u\in X.
\end{equation}
In particular, it holds that
$$
\left\{
  \begin{array}{ll}
\displaystyle   \int_{\R^3}[|\nabla Qu|^2+V(x)|Qu|^2]dx=\|Qu \|^2_X,\\
 \displaystyle   \int_{\R^3}[|\nabla Pu|^2+V(x)|Pu|^2]dx=-\|Pu\|^2_X.
  \end{array}
\right.
$$
Since $\|\cdot\|_X$ and $\|\cdot\|_{H^s(\R^3)}$
are equivalent by $(H_1)$ (see \cite{Kryszewski} for example),
 there exists a constant $S_p>0$ such that
 \begin{equation}\label{Sobolev1}
  \|u\|_{L^p(\R^3)}\leq S_p\|u\|_X,~ \forall p\in[2,2_s^*],
 \end{equation}
 and by \cite{Ziemer}, there holds
 \begin{equation}\label{Sobolev2}
  \|u\|_{L^2(\R^3)}\leq C_t\|u\|_{L^{t,2}(\R^3)}.
 \end{equation}
Moreover, one knows that $X$ and $L^{t,2}(\R^3)$ can be
 compactly imbedded into $L^p_{\text{loc}}(\R^3)$ with $2\leq p<2_s^*:= \frac{6}{3-2s}$
 and $L^2_{\text{loc}}(\R^3)$, respectively.

\subsection{Formulation of problem \eqref{mainequation}}
In this subsection, we assume that
 $\frac{3}{4}<s<1$ and $0<t<1$.
Considering a fixed $u\in X$, the linear functional $\mathcal{L}_u:L^{t,2}(\R^3)\to \R$ is defined by
$$
\mathcal{L}_u(v):=\int_{\R^3}K(x)u^2vdx.
$$
It is bounded since
 the H\"{o}lder's inequality, \eqref{Sobolev1}-\eqref{Sobolev2} and $(H_2)$ show that
\begin{align*}
  |\mathcal{L}_u(v)|&\leq \left(\int_{\R^3}K(x)|u|^4dx\right)^{\frac12} \left(\int_{\R^3}K(x)|v|^2dx\right)^{\frac12}\\
  & \leq\|K\|^{\frac12}_{L^\infty(\R^3)}C_t\|K\|^{\frac12}_{L^{\frac3{4s-3}}(\R^3)}
S_{2_s^*}^{2}\|u\|^2_{X}\|v\|_{L^{t,2}(\R^3)},~\forall v\in L^{t,2}(\R^3).
\end{align*}
So, due to the
  Lax-Milgram theorem, there is a unique $\phi_u\in L^{t,2}(\R^3)$ such that
\begin{align}\label{Lax}
   \int_{\R^3}K(x)u^2vdx=\mathcal{L}_u(v) =(\phi_u,v)_{L^{t,2}(\R^3)}.
\end{align}
From which, we are derived from the
  Plancherel theorem \cite[Theorem 5.3]{Lieb}
that $\phi_u$ is a weak solution of $(I-\Delta)^t\phi=K(x)u^2$ and so, $\phi_u(x)=G_{2t}\ast (K(x)u^2)$.

Letting $v=\phi_u$ in \eqref{Lax}, one has
\[
 \|\phi_u\|_{L^{t,2}(\R^3)}^2
= \int_{\R^3}K(x)\phi_uu^2dx\leq
\|K\|^{\frac12}_{L^\infty(\R^3)}C_t\|K\|^{\frac12}_{L^{\frac3{4s-3}}(\R^3)}
S_{2_s^*}^{2}\|u\|^2_{X}\|\phi_u\|_{L^{t,2}(\R^3)},
\]
which in turn implies that
\begin{equation}\label{nonlocal}
 \int_{\R^3}K(x)\phi_uu^2dx\leq
 \|K\|_{L^\infty(\R^3)}C_t^2\|K\|_{L^{\frac3{4s-3}}(\R^3)}
S_{2_s^*}^{4}\|u\|^4_{X},~ \forall u\in X.
\end{equation}
Inserting $\phi_u$ into \eqref{mainequation}, there holds
\begin{equation}\label{mainequation1}
 (-\Delta)^s u+V(x) u =K(x)\phi_uu+g(x)|u|^{q-2}u,~ x\in\R^3,
\end{equation}
and its corresponding Euler-Lagrange functional $\varphi :X \to\R$ is defined by
\begin{equation}\label{functional1}
 \varphi (u)=\frac{1}{2} \|Qu\|^2_X-\frac{1}{2} \|Pu\|^2_X
-\frac{1}{4}\int_{\R^3}K(x)\phi_uu^2dx-\frac{1}{q}\int_{\R^3}g(x)|u|^qdx.
\end{equation}
In view of \textbf{one} of $(H_3)-(H_5)$, \eqref{decomposition} and \eqref{nonlocal}, it
therefore would be standard to show that $\varphi$ is well-defined on $X$ and belongs to the class of $C^1(X)$
such that
$$
\varphi^{\prime}(u)[v]=\int_{\mathbb{R}^{3}}\left[(-\Delta)^{\frac{s}{2}}u(-\Delta)^{\frac{s}{2}}v+V(x)uv-K(x)
\phi_uuv-g(x)|u|^{q-2}uv\right]dx,~v\in X.
$$
 Obviously, the critical points of $\varphi$ are the weak solutions of problem \eqref{mainequation1}.

\begin{definition}Let $\frac{3}{4}<s<1$ and $0<t<1$.

(l) We call $(u,\phi)\in H^s(\mathbb{R}^3)\times L^{t,2}(\mathbb{R}^3)$ is a weak solution of problem
\eqref{mainequation} if $u$ is a weak solution
of problem \eqref{mainequation1}.

(2) We call $u\in H^s(\mathbb{R}^3)$ is a weak solution of \eqref{mainequation1} if
$$
\int_{\mathbb{R}^{3}}\left[(-\Delta)^{\frac{s}{2}}u(-\Delta)^{\frac{s}{2}}v+V(x)uv+K(x)\phi_{u}uv-g(x)|u|^{q-2}uv\right]dx=0,
$$
for any $v\in H^s(\mathbb{R}^3)$.
\end{definition}

Let us define the variational functional $\Psi:X\to\R$
  by
  $$\Psi(u)=\int_{\R^3}K(x)\phi_uu^2dx
=\int_{\R^3}\big[G_{2t}\ast (K(x) u^2)\big]K(x) u^2dx.$$
In the following, we collect some important properties for $\Psi$ as follows.

\begin{lemma}\label{Psi} Let $\frac{3}{4}<s<1$ and $0<t<1$
as well as $(H_1)-(H_2)$, then the
 following conclusions hold true:
 \begin{itemize}
  \item[\emph{(i)}] $\Psi(u)\geq 0$ and $\Psi(\theta u)=\theta^4\Psi(u)$ for any
  $u\in X$ and $\theta>0$;
 \item[\emph{(ii)}] If $u_n\rightharpoonup u$ in $X$, $u_n\to u$ in $L_{\text{\emph{loc}}}^p(\R^3)$ with $2<p<2_s^*$
    and $u_n\to u$ a.e. in $\R^3$ as $n\to\infty$, then, going to a subsequence if necessary,
   $$\Psi(u_n)\to \Psi(u) ~\text{\emph{and}}~ \Psi^\prime(u_n)[\psi]\to \Psi^\prime(u)[\psi]$$ for every $\psi\in X$ as
   $n\to\infty$;
  \item[\emph{(iii)}] For any fixed constant $\alpha\in(0,1)$ and every finite dimension
  subspace $Z_0\subset Z$, there exists a constant
   $C_\alpha>0$ such that for each $u\in Y\bigoplus Z_0$ with
   $\|u\|_{X}=1$ there holds
   \[
    \| Qu\|_{X}\geq \sin(\arctan \alpha)
    \Longrightarrow  \Psi(u)\geq C_\alpha.
   \]
\end{itemize}
\end{lemma}

\begin{proof}
(i) In consideration of $K(x)>0$ and $G_{2t}(x)>0$ for all in $x\in\R^3$ by $(H_2)$ and
\eqref{Yukawa}, respectively,
we immediately derive that $\Psi(u)\geq 0$
according to its definition.
Moreover, for all $\theta>0$, it is simple to see that
\begin{align*}
 \Psi(\theta u)  & =\int_{\R^3}\big[G_{2t}\ast (K(x) (\theta u)^2)\big]K(x) (\theta u)^2dx \\
    & =\theta^4\int_{\R^3}\big[G_{2t}\ast (K(x) u^2)\big]K(x) u^2dx=\theta^4\Psi( u),~\forall u\in X,
\end{align*}
showing the Point-(i).

(ii) We can rewrite $\Psi(u_n)-\Psi(u)$ by
\begin{align*}
  \Psi(u_n)-\Psi(u) & =\int_{\R^3}K(x)(\phi_{u_n}-\phi_u) u_n^2dx
+\int_{\R^3}K(x)\phi_u(u_n^2-u^2)dx \\
   &:=I_n^1+I_n^2.
\end{align*}
Recalling $\phi_{u_n}-\phi_{u_n}=G_{2t}\ast (K(x)(u_n^2-u^2))$
and $\|G_{2t}\|_{L^1(\R^3)}=1$, we then apply the Young's inequality with respect to the convolution operator and
$K\in L^{\frac{3}{4s-3}}(\R^3)$ in $(H_2)$ to get
$$\begin{gathered}
  \|\phi_{u_n}-\phi_{u_n}\|_{L^2(\R^3)} \leq \left(\int_{\R^3}K^2(x)|u_n -u |^2|u_n+u|^2dx\right)^{\frac12} \hfill\\
 \ \ \ \    \leq \|K\|^{\frac12}_{L^\infty(\R^3)}
\left(\int_{\R^3}K (x)|u_n -u |^4dx\right)^{\frac14}
\left(\int_{\R^3}K (x) |u_n+u|^4dx\right)^{\frac14}\hfill\\
 \ \ \ \    \leq \|K\|^{\frac34}_{L^\infty(\R^3)}S_4\|u_n\|_X
\left(\int_{\R^3}K (x)|u_n -u |^4dx\right)^{\frac14}=o_n(1),
\hfill\\
\end{gathered}$$
where the last equality follows the generalized Vitali's Convergence theorem.
Thus,
\begin{align*}
 |I_n^1| &\leq\left(\int_{\R^3}K(x)(\phi_{u_n}-\phi_u)^2dx\right)^{\frac12}
\left(\int_{\R^3}K(x)|u_n|^4dx\right)^{\frac12} \\
    & \leq \|K\| _{L^\infty(\R^3)}S_4^2\|u_n\|_X^2 \|\phi_{u_n}-\phi_{u_n}\|_{L^2(\R^3)}
 =o_n(1).
\end{align*}
Similarly, we can also obtain that
\begin{align*}
 |I_n^2| &\leq\left(\int_{\R^3}K(x)\phi^2_udx\right)^{\frac12}
\left(\int_{\R^3}K(x)|u_n+u|^2|u_n-u|^2dx\right)^{\frac12} \\
    & \leq  \|K\|^{\frac34}_{L^\infty(\R^3)} C_t\|\phi_u\|_{L^{t,2}(\R^3)}^2
S_4\|u_n+u\|_X
\left(\int_{\R^3}K(x)|u_n-u|^4dx\right)^{\frac14}\\
&=o_n(1).
  \end{align*}
Combining the above two facts, we derive the proof of the first part. The remaining part is easier, so we omit it here.

(iii)
Let us define a constant $\gamma:= \sin(\arctan \alpha)\in(0,1)$ and a set
$$
\Upsilon^{\alpha}:=\left\{v\in Y\bigoplus Z_0:
\|v\|_{X}=1~ \text{and}~
\|Qv\|_{X}\geq\gamma\right\}.$$
Due to the definition of $u$, one sees that
$\Psi(u)\geq \inf\limits_{v\in \Upsilon^{\alpha}}\Psi(v)
:= C_\alpha$. Using Point-(i), there holds that $C_\alpha\geq0$.
So, to finish the proof,
it suffices to conclude that $C_\alpha>0$.
Suppose, by contradiction, that $C_\alpha=0$.
Then, there exists a sequence $\{v_n\}\subset \Upsilon^{\alpha}$
such that $\Psi(v_n)\to 0$ as $n\to\infty$.
Since $\|v_n\|_X\equiv1$,
up to a subsequence if necessary,
there is $v\in X$ such that $v_n\rightharpoonup v$ in $X$
and $Qv_n\to Qv$
because $\{Qv_n\}\subset Z_0$ with $\dim Z_0<+\infty$. Hence,
$\|Qv\|_{X}\geq\gamma>0$ and
$\Psi(v)\leq \liminf_{n\to\infty}\Psi(v_n)=0$ yielding that $v\equiv0$, a contradiction to (ii).
The proof is completed.
\end{proof}

Next, we prove that the variational functional $\varphi$ satisfies the $(PS)_c$ condition.

\begin{lemma}\label{PSc} Let $\frac{3}{4}<s<1$ and $0<t<1$.
Suppose that $(H_1)-(H_2)$ and \textbf{one} of $(H_3)-(H_5)$ hold, then $\varphi$
satisfies the $(PS)_c$ condition.
\end{lemma}

\begin{proof}
Assume that there is a sequence $\{u_n\}\subset X$ satisfies $\varphi(u_n)\to c$ and
$\varphi^\prime(u_n)\to 0$ as $n\to\infty$, then
\begin{align}\label{bounded1}
\nonumber  c+1+\|u_n\|_X & \geq \varphi(u_n)-\frac{1}{2} \varphi^\prime(u_n)[u_n] \\
    & =\frac{1}{4}\Psi (u_n)+\frac{q-2}{2q}\int_{\R^3}g(x)|u_n|^qdx.
\end{align}
To show that $\|u_n\|_X$ is uniformly bounded in $n\in \mathbb{N}$, we split it into two cases.

\textbf{Case 1:} the assumption $(H_3)$ holds.\\
In this case, that is,
$g\in L^{q_0}(\R^3)$ with
   $q_0=\frac{6}{6-q(3-2s)}$ and $1<q<\frac43$,
we can adopt \eqref{bounded1} together with \eqref{Sobolev1} to have that
\begin{align*}
  \Psi (u_n) & \leq 4(c+1+\|u_n\|_X)+\frac{q}{2(2-q)}\left(\int_{\R^3}|g|^{q_0}dx\right)^{\frac1{q_0}}
\left(\int_{\R^3}|u_n|^{2^*}dx\right)^{\frac{q}{2_s^*}} \\
    & \leq C(1+\|u_n\|_X+\|u_n\|^q_X).
\end{align*}
 Denoting $u_n=Pu_n+Qu_n:= y_n+z_n$
with $y_n\in Y$ and $z_n\in Z$, then
\begin{align}\label{bounded2}
\nonumber  \big|\Psi^\prime(u_n)[y_n]\big| &=\left|\int_{\R^3}K(x)\phi_{u_n}u_ny_ndx\right|
\leq \Psi^{\frac{1}{2}}(u_n)
\bigg(\int_{\R^3}K(x)\phi_{u_n}y_n^2dx\bigg)^{\frac{1}{2}}\\
\nonumber &
\leq \Psi^{\frac{1}{2}}(u_n)
\bigg(\int_{\R^3}K(x)\phi_{u_n}^2dx\bigg)^{\frac{1}{4}}
\bigg(\int_{\R^3}K(x)y_n^4dx\bigg)^{\frac{1}{4}} \\
\nonumber &\leq  \|K\|^{\frac{1}{2}}_{L^\infty(\R^3)} C_t^{\frac{1}{2}}S_4\Psi^{\frac{3}{4}}(u_n)\|y_n\|_X\\
 \nonumber   & \leq C(1+\|u_n\|_X+\|u_n\|^q_X)^{\frac{3}{4}}\|y_n\|_X\\
   &
   \leq C(1+\|u_n\|_X+\|u_n\|^q_X)^{\frac{3}{4}}\|u_n\|_X.
\end{align}
Thereby, for $n\in\mathbb{N}$ large, combining $\|u_n\|_X\geq
\|y_n\|_X\geq -\varphi^\prime(u_n)[y_n]$
and \eqref{bounded2}, we obtain
\begin{align*}
 \|y_n\|_X^2 &=-\varphi^\prime(u_n)[y_n]-\Psi^\prime(u_n)[y_n]-\int_{\R^3}g(x)|u_n|^{q-2}u_ny_ndx  \\
   & \leq \|u_n\|_X+C(1+\|u_n\|_X+\|u_n\|^q_X)^{\frac{3}{4}}\|u_n\|_X
+C\|u_n\|_X^{q-1}\|y_n\|_X.
\end{align*}
 Similarly, we deduce that $$\|z_n\|_X^2\leq \|u_n\|_X+C(1+\|u_n\|_X+\|u_n\|^q_X)^{\frac{3}{4}}\|u_n\|_X
 +C\|u_n\|_X^{q-1}\|z_n\|_X.$$
 Recalling the fact that $\|u_n\|_X^2=\|y_n\|_X^2+\|z_n\|_X^2$, we know that
 $\|u_n\|_X$ is uniformly bounded since $2>1+\frac{3}{4}q$ which is $q<\frac43$ by $(H_3)$.

 \textbf{Case 2:} either the assumption $(H_4)$ or $(H_5)$ holds.\\
 Obviously, both $(H_4)$ and $(H_5)$ indicate that $(q-2)g(x)>0$ a.e. in $\R^3$.
 Therefore, it follows from \eqref{bounded1}
   that
$$\Psi(u_n)\leq C(1+\|u_n\|_X)$$ and $$0\leq (q-2)\int_{\R^3}g(x)|u_n|^qdx\leq C(1+\|u_n\|_X).$$
It is similar to \eqref{bounded2} that
 $|\Psi^\prime(u_n)[y_n]|\leq C(1+\|u_n\|_X)^{\frac{3}{4}}\|u_n\|_X$ and then
 \[
 \|y_n\|_X^2\leq \|u_n\|_X+C(1+\|u_n\|_X)^{\frac{3}{4}}\|u_n\|_X
+C (1+\|u_n\|_X)^{\frac{q-1}{q}}\|y_n\|_X.
\]
Analogously, one has
$$\|z_n\|_X^2\leq \|u_n\|_X+C(1+\|u_n\|_X)^{\frac{3}{4}}\|u_n\|_X
+C (1+\|u_n\|_X)^{\frac{q-1}{q}}\|z_n\|_X.$$
So, we still can derive
$\|u_n\|_X$ is uniformly bounded since $2>1+\frac{3}{4}$.

Based on the above discussions,
we can conclude that $\|u_n\|_X$ is uniformly bounded in $n\in \mathbb{N}$.
Moreover, one deduces that $\|y_n\|_X$ and $\|z_n\|_X$ are uniformly bounded in $n\in \mathbb{N}$.
Passing to subsequences if necessary,
there exist two functions $y\in Y$ and $z\in Z$ such that, as $n\to\infty$,
$$
\left\{
  \begin{array}{ll}
    y_n\rightharpoonup y~\text{and}~z_n\rightharpoonup z, &\text{in}~X, \\
    y_n\to y~\text{and}~z_n\rightharpoonup z, & \text{in}~L_{\text{{loc}}}^p(\R^3)~\text{with}~2\leq p<2_s^*, \\
    y_n\to y~\text{and}~z_n\rightharpoonup z, & \text{a.e. in}~\R^3
  \end{array}
\right.
$$
Define $u:=y+z\in X$, then we aim to show that $y_n\to y$ in $X$ along a subsequence.
Actually, we can argue as the calculations in \eqref{bounded2} to obtain that
$$(\Psi^\prime(u_n)-\Psi^\prime(u))[y_n-y]=o_n(1)$$ as $n\to\infty$.
Since $g\in L^{q_0}(\R^3)$, it follows from the generalized Vitali's Convergence theorem again that
$$\begin{gathered}
\left|\int_{\mathbb{R}^N}g(x)(|u_n|^{q-2}u_n-|u|^{q-2}u)(y_n-y)dx\right|\hfill\\
\ \ \ \  \leq \left(\int_{\mathbb{R}^N}|g(x)|\big||u_n|^{q-2}u_n-|u|^{q-2}u\big|^{\frac{q}{q-1}}dx\right)^{\frac{q-1}q}
\left(\int_{\mathbb{R}^N}|g(x)||y_n-y|^qdx\right)^{\frac1q}\hfill\\
\ \ \ \   =o_n(1)
\hfill\\
\end{gathered}$$
as $n\to\infty$. Combining the above two formulas and $\varphi'(u_n)\to0$, there holds
\begin{align*}
  o_n(1) &=\big(\varphi'(u_n)-\varphi'(u)\big) [y_n-y]\\
    &=-\|y_n-y\|_X^2-\big(\Psi'(u_n)-\Psi'(u)\big) [y_n-y]\\
&\ \ \ \  -\int_{\mathbb{R}^N}g(x)(|u_n|^{q-2}u_n-|u|^{q-2}u)(y_n-y)dx\\
    &=-\|y_n-y\|_X^2=o_n(1)
\end{align*}
which indicates the desired result. Analogously, we can also derive that
$\|z_n-z\|_X=o_n(1)$. So, one sees that $u_n\to u$ in $X$ finishing the proof.
\end{proof}

% The title of your section 3:

 \section{Proofs of the main results}\label{Sec3}\

In this section, we are going to exhibit the proofs of Theorems \ref{maintheorem1},
\ref{maintheorem2} and \ref{maintheorem3} in detail.
For this goal, we need to introduce the
following improved Fountain theorem developed by Gu and Zhou in \cite{Gu}, that is,

\begin{proposition}\label{Gu}
Let $\Phi\in C^1(X,\R)$ be an even functional satisfying $(PS)_c$ condition and
$\nabla \Phi$ be weakly sequentially continuous, for all
$k\in \mathbb{N}$ with $X=Y_k\bigoplus Z_k$, if there exist two constants $\rho_k>r_k>0$ such that
\begin{itemize}
  \item[$({A}_1)$]  $d_k:= \sup\limits_{u\in Y_k,\|u\|_X\leq \rho_k}\Phi(u)<+\infty$;
  \item[$({A}_2)$]  $a_k:= \sup\limits_{u\in Y_k,\|u\|_X= \rho_k}\Phi(u)<\inf\limits_{u\in Z_k,\|u\|_X\leq r_k}\Phi(u)$;
  \item[$({A}_3)$]  $b_k:= \inf\limits_{u\in Z_k,\|u\|_X= r_k}\Phi(u)\to+\infty$ as $k\to+\infty$;
  \item[$({A}_4)$]  For any $\sigma>0$, there is a constant $C_\sigma>0$ such that $\sup\limits_{\|u\|_\tau<\sigma}\Phi(u)\leq C_\sigma<+\infty$,
where
\begin{equation}\label{tau}
\|u\|_{\tau}:=\max\bigg\{\sum_{j=0}^{\infty}\frac{1}{2^{j+1}}|\langle Pu,e_{j}\rangle|,\|Qu\|_X\bigg\}\quad\mathrm{for}~u\in X,
\end{equation}
with $\{e_j\}_{j=1}^\infty$ denoting the normal orthogonal basis of $Y$.
\end{itemize}
Then, the functional $\Phi$ has a sequence of critical points $\{u^{k_m}\}$ such that
$$\lim\limits_{m\to\infty}\|u^{k_m}\|_X=+\infty.$$
\end{proposition}

Let us recall the decomposition on $X=Y\bigoplus Z$ in Section \ref{Sec2},
we could suppose that the two sequences $\{e_j\}_{j=1}^\infty$ and $\{f_j\}_{j=1}^\infty$
are normal orthogonal basis of $Y$ and $Z$, respectively. For any
$k\in \mathbb{N}$, define
\[
Y_k:= Y\bigoplus\left(\bigoplus_{j=1}^k\R f_j\right)~ \text{and}~
Z_k:= \overline{\bigoplus_{j=k+1}^\infty\R f_j}.
\]

In order to apply the Proposition \ref{Gu} successfully, we shall regard the variational
functional $\Phi$ as $\varphi$ which corresponds to Eq. \eqref{functional1}.
From now on until the end of this section, we shall always suppose that $\frac34<s<1$ and $0<t<1$
and do not mention them any longer.

First of all, we show that the
functional $\varphi$ satisfies the condition $(A_1)$.

\begin{lemma}\label{linking1}
Assume $(H_1)-(H_2)$ and \textbf{\emph{one}} of $(H_3)-(H_5)$ hold,
then there exists $\rho_k>0$ such that
$$
d_k:= \sup\limits_{u\in Y_k,\|u\|_X\leq \rho_k}\Phi(u)<+\infty.
$$
 \end{lemma}

 \begin{proof}
The proof is standard since $\varphi$ maps a bounded set into a bounded set,
so we omit it here.
\end{proof}

Then, we verify that the
functional $\varphi$ satisfies the conditions $(A_2)$ and $(A_3)$.

\begin{lemma}\label{linking2}
Assume $(H_1)-(H_2)$ and \textbf{\emph{one}} of $(H_3)-(H_5)$ hold,
then there exist $\rho_k>r_k>0$ such that
$a_k<\inf\limits_{u\in Z_k,\|u\|_X\leq r_k}\varphi(u)$
and $b_k\to+\infty$ as $k\to+\infty$.
 \end{lemma}

 \begin{proof}
 Firstly, to derive the first part,
we just need to show that
$$
a_k:= \sup_{u\in Y_k,\|u\|_X= \rho_k}\varphi(u)\to-\infty
$$
 as $\rho_k\to+\infty$. Either $(H_3)$ or $(H_4)$ holds true, for all $u\in X$, there holds
$$
-\frac1q\int_{\R^3}g(x)|u|^qdx\leq C\|u\|_X^q,~1<q<2,
$$
for some positive constant $C$ which is independent of $u$. If $(H_5)$ holds true, because $g(x)>0$
for all $x\in\R^3$, given $u\in X$, it has that
$$
-\frac1q\int_{\R^3}g(x)|u|^qdx\leq 0.
$$
In summary, for any fixed $u=Qu+Pu\in X$, we can always deduce that
 \begin{equation}\label{HHH}
 \varphi(u)\leq \frac{1}{2}\|Qu\|^2_X-\frac{1}{2}\|Pu\|^2_X
 -\frac{1}{4}\Psi(u)+C\|u\|_X^q,~ \text{where}~ 1<q<2.
 \end{equation}
Now, we begin verifying that $a_k\to-\infty$ as $\rho_k\to+\infty$.
Given a $u\in Y_k=Y\bigoplus \left(\bigoplus\limits_{j=1}^k\R f_j\right)$
with $\|u\|_X=\rho_k$.

If
 $u\in Y$, then $Qu=0$ and $Pu=u$, it follows from \eqref{HHH} that
$$
\varphi(u)\leq -\frac{1}{2}\rho_k^2+C\rho_k^q\to-\infty
$$
as
 $\rho_k\to+\infty$ since $q<2$ in \eqref{HHH}.

 If $u=y+z$ with $y\in Y$ and $z\in \bigoplus\limits_{j=1}^k\R f_j$,
for the constant $\alpha\in(0,1)$ given by Lemma \ref{Psi}-(iii),
we shall distinguish the proof two cases:
\[
\text{(i)}~  \|Q u\|_{X}/
\|P u\|_{X}<\alpha
~ \text{and}~
\text{(ii)}   ~\|Q u\|_{X}/
\|P u\|_{X}\geq\alpha.
\]
If (i) occurs, then $\|u\|_{X}^2=\|Qu\|_{X}^2+\|Pu\|_{X}^2<(1+\alpha^2)
\|P u\|_{X}^2$ which in turn shows that
$$
\|Qu\|_{X}^2=\|u\|_{X}^2-\|Pu\|_{X}^2<\frac{\alpha^2}{1+\alpha^2}\|u\|_{X}^2.
$$
Combining $\alpha\in(0,1)$ in Lemma \ref{Psi}-(iii) and \eqref{HHH} with $q<2$, we obtain
\begin{align*}
  \varphi (u) &\leq \frac{\alpha^2}{2(1+\alpha^2)} \|u\|_{X}^2-\frac{1}{2(1+\alpha^2)}
\|u\|_{X}^2+C\|u\|_{X}^q \\
   &= -\frac{1-\alpha^2}{2(1+\alpha^2)}\rho_k^2+C\rho_k^q\\
   &
   \to-\infty~ \text{as}~ \|u\|_{X}=\rho_k\to+\infty.
\end{align*}
If (ii) occurs, we set $v:=
\frac{u}{\|u\|_{X}}$ and so $\|v\|_X=1$.
Moreover, one has that
$$\|Qv\|_X=\frac{\|Qu\|_X}{\|u\|_X}
=\sin\left(
 \arctan \frac{\|Qu\|_{X}}{\|Pu\|_{X}}
 \right).
$$
In this case, one deduces $\|Qv\|_X\geq \sin(\arctan \alpha)$.
Accordingly, $v\in Y\bigoplus \left(\bigoplus\limits_{j=1}^k\R f_j\right)$ with
$\dim \left(\bigoplus\limits_{j=1}^k\R f_j\right)<+\infty$, as a consequence of Lemma
  \ref{Psi}-(i) and (iii), we obtain $\Psi(u)\|u\|_X^{-4}=\Psi(v)\geq C_\alpha>0$.
  Then, adopting \eqref{HHH} again, we arrive at
\begin{align*}
 \varphi (u) & \leq \frac{1}{2}\|u\|_{X}^2-\frac{C_\alpha}{4} \|u\|_X^{4}
 +C\|u\|_{X}^q \\
    & =\frac{1}{2}\rho_k^2-\frac{C_\alpha}{4} \rho_k^{4}
 +C\rho_k^q\\
 &
\to-\infty~ \text{as}~ \|u\|_{X}=\rho_k\to+\infty.
\end{align*}
Therefore, we always have $a_k\to-\infty$ as $\rho_k\to+\infty$ and the first part concludes.

Secondly, we
 claim that $$\beta_k^1:= \sup_{u\in Z_k:\|u\|_X=1}\Psi(u)\to0$$
  as $k\to+\infty$.
Indeed, according to the definition of $Z_k$, one simply has $0<\beta_{k+1}^1\leq \beta_k^1$
 for any $k\in \mathbb{N}$. Then, there exists a constant $\beta\geq0$
 such that $\beta_k^1\to\beta$ as $k\to+\infty$. By the
 definition of $\beta^1_k$, there is a $u_k\in Z_k$ with
 $\|u_k\|_X=1$ such that $\Psi(u_k)\geq \beta_k^1/2$.
 Since $\{u_k\}\subset Z_k=\overline{\bigoplus\limits_{j=k+1}^\infty\R f_j}$,
 one concludes that $u_k\rightharpoonup 0$ in $X$ as $k\to+\infty$.
 In view of Lemma \ref{Psi}-(ii), passing to a subsequence if necessary, we derive
 $\Psi(u_k)\to0$ as $k\to+\infty$. Hence,
we must have $\beta=0$ and the claim holds true.

 Finally, to prove $b_k\to+\infty$ as $k\to+\infty$, we split it into two cases.

\textbf{Case 1:} either the assumption $(H_3)$ or $(H_4)$ holds true.\\
In this case, combining $q<2$ and $\beta_k^1\to0$ as $k\to+\infty$,
there holds
$$
 \frac{1}{4}(\beta_k^1)^{\frac{q-2}{2}}-\frac{C}{q}\geq \frac{1}{8}~\text{for some sufficiently large}~k\in \mathbb{N},
$$
where $C>0$ is a constant independent of $k$.
As a consequence, for any $u\in Z_k$ with $\|u\|_X=r_k=(\beta_k^1)^{-\frac{1}{2}}$,
by means of $g\in L^{q_0}(\R^3)$,
we have
\begin{align*}
  \varphi(u) &\geq \frac{1}{2}\|u\|^2_X-\frac{\beta_k^1}{4}\|u\|^4_X-\frac{C}{q}\|u\|^q_X\\
  &
  =\bigg(\frac{1}{4}\|u\|^{2-q}_X-\frac{C}{q}\bigg)\|u\|^q_X
  + \frac{1}{4}(1-\beta_k^1\|u\|^2_X)\|u\|^2_X\\
  &= \bigg(\frac{1}{4}\|u\|^{2-q}_X-\frac{C}{q}\bigg)\|u\|^q_X
  \geq \frac{1}{8}\|u\|^q_X=\frac{1}{8}(\beta_k^1)^{-\frac{q}{2}}:=\frac{1}{8}r_k^q,
\end{align*}
for some sufficiently large $k\in \mathbb{N}$. Clearly,
$r_k=(\beta_k^1)^{-\frac{1}{2}}\to+\infty$ as $k\to+\infty$. Hence, it holds that
$b_k\to+\infty$ as $k\to+\infty$.

\textbf{Case 2:} the assumption $(H_5)$ holds true.\\
Since $g(x)>0$ a.e. for $x\in\R^3$ and $g\in L^{q_0}(\R^3)$, for any $k\in \mathbb{N}$, we can verify that
$$\beta_k^2:=\sup_{u\in Z_k:\|u\|_X=1}\int_{\R^3}g(x)|u|^qdx\to0.$$
 Indeed, one can easily get that $0<\beta_{k+1}^2\leq \beta_k^2$
 for any $k\in \mathbb{N}$ and so there is a constant $\hat{\beta}\geq0$
 such that $\beta_k^2\to\hat{\beta}$ as $k\to+\infty$.
By the
 definition of $\beta^2_k$, there is a $u_k\in Z_k$ with
 $\|u_k\|_X=1$ such that $\int_{\R^3}g(x)|u_k|^qdx\geq \beta_k^2/2$.
Recalling $\{u_k\}\subset Z_k=\overline{\bigoplus\limits_{j=k+1}^\infty\R f_j}$,
 one concludes that $u_k\rightharpoonup 0$ in $X$ as $k\to+\infty$.
Due to $g\in L^{q_0}(\R^3)$, one immediately has that
$\int_{\R^3}g(x)|u_k|^qdx\to0$ and so $\hat{\beta}=0$.
Therefore, for all
$u\in Z_k$ with $\|u\|_X=r_k=\min\bigg\{
  \sqrt{\frac{1}{2\beta_k^1}} ,
 \bigg(\frac{q}{8\beta_k^2}\bigg)^{\frac{1}{q-2}}\bigg\}$, we have
\begin{align*}
 \varphi(u) &\geq
 \frac{1}{4}\|u\|^2_X+\bigg(\frac{1}{8}-\frac{\beta_k^1}{4}\|u\|^2_X\bigg)\|u\|^2_X
 +\bigg(\frac{1}{8}-\frac{\beta_k^2}{q}\|u\|^{q-2}_X\bigg)\|u\|^2_X\\
   &\geq
 \frac{1}{4}\|u\|^2_X =\frac{1}{4}\min\bigg\{
  \frac{1}{2\beta_k^1} ,
 \bigg(\frac{q}{8\beta_k^2}\bigg)^{\frac{2}{q-2}}\bigg\}\\
 &
 := \frac{1}{4}r_k^2.
\end{align*}
It is easy to see that $r_k \to+\infty$ as $k\to+\infty$ since $q>2$
in this case.
So,
$b_k\to+\infty$ as $k\to+\infty$. The proof of this lemma is completed.
 \end{proof}

Finally, the condition $(A_4)$ will be proved for the
functional $\varphi$ as follows.

\begin{lemma}\label{linking3}
Assume $(H_1)-(H_2)$ and \textbf{\emph{one}} of $(H_3)-(H_5)$ hold,
then for any $\sigma>0$, there is a constant $C_\sigma>0$ such that $\sup\limits_{\|u\|_\tau<\sigma}\varphi(u)\leq C_\sigma<+\infty$.
 \end{lemma}

 \begin{proof}
Since $u\in X= Y\bigoplus Z$ with $Z= Y^\perp$, we may set $u=Qu+Pu$ with $Qu\in {Z}$ and
 $Pu\in Y$. Then,
using a very similar calculations in \eqref{HHH},
we reach
$$
\begin{aligned}
\varphi(u)& =\frac{1}{2}\|Qu\|^2_X-\frac{1}{2}\|Pu\|^2_X
-\frac{1}{4}\Psi(u)-\frac{1}{q}\int_{\mathbb{R}^{N}}g(x)|u|^{q}\:dx \\
&\leq\frac{1}{2}\|Qu\|^2_X-\frac{1}{2}\|Pu\|^2_X+C\|u\|_X^q\\
&\leq\frac{1}{2}\|Qu\|^2_X+C\|Qu\|_X^{q}-\frac{1}{2}\|Pu\|^2_X+C\|Pu\|_X^{q}.
\end{aligned}
$$
Since $q<2$ in \eqref{HHH}, $-\frac{1}{2}\|Pu\|_X^2+C\|Pu\|_X^q$ is bounded from above.
By \eqref{tau}, we have $\|Qu\|_X\leq\|u\|_\tau\leq\sigma$, thus there exists a $C_\sigma<\infty$ such
that $\sup\limits_{\|u\|_\tau\leq\sigma}\varphi(u)<C_{\sigma}$.
The proof is completed
\end{proof}

With Lemmas \ref{linking1}, \ref{linking2} and \ref{linking3} in hands,
we derive that the variational functional $\varphi$ has infinitely many nontrivial critical points.

\begin{lemma}\label{solution}
Assume $(H_1)-(H_2)$ and \textbf{\emph{one}} of $(H_3)-(H_5)$ hold,
then $\varphi$ has infinitely many nontrivial critical points $\{u_n\}\subset H^s(\R^3)$
such that $\lim\limits_{n\to\infty}\|u_n\|_{H^s(\R^3)}=+\infty$.
 \end{lemma}

 \begin{proof}
Let $\varphi=\Phi$ be as in Proposition \ref{Gu}, and clearly $\varphi$ is even.
Owing to Lemmas \ref{Psi}-(ii) and \ref{PSc}, we clearly know that $\varphi$
satisfies $(PS)_c$ condition and
$\nabla \varphi$ is weakly sequentially continuous.
On the other hand, we have validated
the conditions $(A_1)$, $(A_2)$, $(A_3)$ and $(A_4)$ for $\varphi$
in Lemmas \ref{linking1}, \ref{linking2} and \ref{linking3}, respectively.
As a consequence of Proposition \ref{Gu}, we can finish the proof of this lemma.
 \end{proof}

At this stage, we can conclude the proofs of Theorems \ref{maintheorem1}, \ref{maintheorem2}
and \ref{maintheorem3} depending on Lemma \ref{solution}.
%%%%%%%%%%%%%%%%%%%%%%%%%%%%%%%%%%%%%%%%%%%%%%%%%%%%%%
%          AI TOOLS, USE AND LOCATION
%%%%%%%%%%%%%%%%%%%%%%%%%%%%%%%%%%%%%%%%%%%%%%%%%%%%%%
%We follow COPE's guidelines and policies regarding the use of Artificial Intelligence (AI) tools. COPE Policy on AI tools can be found at https://publicationethics.org/cope-position-statements/ai-author.

%Authors using AI tools in the writing of a manuscript, production of images or graphical elements of the paper, or in the collection and analysis of data, must be transparent in disclosing in this section how the AI tool was used and which tool was used. Authors are fully responsible for the content of their manuscript, even those parts produced by an AI tool, and are thus liable for any breach of publication ethics. - COPE

%Disclosure instructions

%If there is nothing to disclose, there is no need to add a declaration, otherwise please declare.

%\section*{Use of AI tools declaration}
%The author(s) declare(s) they have used Artificial Intelligence (AI) tools in the creation of this article.
%AI tools used:
%How were the AI tools used?
%Where in the article is the information located?

%The acknowledgments section should not be numbered.
\section*{Acknowledgments}
The work was partially carried out while the third author was visiting Professor
Juncheng Wei at University of British Columbia and he acknowledges the
warm hospitality from Department of Mathematics.

%%%%%%%%%%%%%%%%%%%%%%%%%%%%%%%%%%%%%%%%%%%%%%%%%%%%%%
%          7. REFERENCES SECTION
%%%%%%%%%%%%%%%%%%%%%%%%%%%%%%%%%%%%%%%%%%%%%%%%%%%%%%

%       READ THIS SECTION CAREFULLY

% Each of the references below MUST be cited in your article above. Do not include references that are not cited in your article.

% Follow the examples below carefully. We strongly suggest that you copy and paste your reference information directly into our examples.

% List all references in alphabetical order according to the first author's last name.

% Verify each URL works correctly and can be accessed properly. Your URL links should be to reputable websites. The command line for a website link begins with: \url{ }

% Do not add MR or DOI numbers to your references. AIMS production staff will add this information.

% Using BibTex is not recommended but can be handled.

\medskip
% The information below will be filled in by AIMS production staff.
Received xxxx 20xx; revised xxxx 20xx; early access xxxx 20xx.
\medskip

\end{document}